\documentclass{gen-j-l}
\usepackage{amssymb}
\usepackage{amsfonts}

\newtheorem{theorem}{Theorem}[section]
\newtheorem{lemma}[theorem]{Lemma}
\theoremstyle{definition}

\theoremstyle{remark}
\newtheorem{remark}[theorem]{Remark}
\numberwithin{equation}{section}
\theoremstyle{plain}
\newtheorem{acknowledgement}{Acknowledgement}

\newtheorem{corollary}{Corollary}

\newtheorem{proposition}{Proposition}

\copyrightinfo{2001}{enter name of copyright holder}
\input{tcilatex}

\begin{document}
\title[G-Dedekind domains]{Revisiting G-Dedekind domains}
\author{M. Zafrullah}
\address{Department of Mathematics, Idaho State University, Pocatello, 83209
ID}
\email{mzafrullah@usa.net}
\thanks{I am thankful to my friends who made me rise from the dead.}
\dedicatory{Dedicated to fairness}

\begin{abstract}
Let $R$ be an integral domain with $qf(R)=K$ and let $F(R)$ be the set of
nonzero fractional ideals of $R.$ Call $R$ a dually compact domain (DCD) if
for each $I\in F(R)$ the ideal $I_{v}=(I^{-1})^{-1}$ is a finite
intersection of principal fractional ideals. We characterize DCDs and show
that the class of DCDs properly contains various classes of integral
domains, such as Noetherian, Mori and Krull domains. In addition we show
that a Schreier DCD is a GCD domain with the property that for each $A\in
F(R)$ the ideal $A_{v}$ is principal. We show that a domain $R$ is
G-Dedekind (i.e. has the property that $A_{v}$ is invertible for each $A\in
F(R)$) if and only if $R$ is a DCD satisfying the property $\ast :$ for all
pairs of subsets $\{a_{1},...,a_{m}\},\{b_{1},...b_{n}\}\subseteq
K\backslash \{0\},$ $(\cap _{i=1}^{m}(a_{i})(\cap _{j=1}^{n}(b_{j}))=\cap
_{i,j=1}^{m,n}(a_{i}b_{j})$. We discuss what the appropriate name for
G-Dedekind domains and related notions should be. We also make some
observations about how the DCDs behave under localizations and polynomial
ring extensions.
\end{abstract}

\maketitle

\section{Introduction}

Let $R$ be an in integral domain with quotient field $K$ and let $F(R)$ be
the set of nonzero fractional ideals of $R.$ Call $R$ a dually compact
domain (DCD) if for each set $\{a_{\alpha }\}_{\alpha \in I}$ $\subseteq
K\backslash \{0\}$ with $\cap a_{\alpha }R\neq (0)$ there is a finite set of
elements $\{x_{1},...x_{r}\}\subseteq K\backslash \{0\}$ such that $\cap
a_{\alpha }R=\cap _{i=1}^{r}x_{i}R,$ or equivalently for each $I\in F(R),$
the ideal $I_{v}=(I^{-1})^{-1}$ is a finite intersection of principal
fractional ideals of $R$. We characterize DCDs (in section \ref{S3}) and
show that the class of DCDs properly contains various classes of integral
domains of import such as Noetherian, Mori and Krull domains, in section \ref%
{S4}. In addition we show that a pre-Schreier DCD is a GCD domain with the
property that for each $A\in F(R)$ the ideal $A_{v}$ is principal. (Here $R$
is pre-Schreier if for all $x,y,z\in R\backslash \{0\}$ $x|yz$ $\Rightarrow $
$x=rs,$ with $r,s\in R$ such that $r|y$ and $s|z.)$ We show that a domain $R$
is a G-Dedekind domain (i.e. has the property that $A_{v}$ is invertible for
each $A\in F(R)$) if and only if $R$ is a DCD satisfying the property $\ast
: $ for all pairs of subsets $\{a_{1},...,a_{m}\},\{b_{1},...b_{n}\}%
\subseteq K\backslash \{0\},$ $(\cap _{i=1}^{m}(a_{i})(\cap
_{j=1}^{n}(b_{j}))=\cap _{i,j=1}^{m,n}(a_{i}b_{j})$ from \cite{Z pres}. We
discuss in section \ref{S2} what names for G-Dedekind domains and related
notions should be appropriate. We also make some observations about how the
DCDs behave under localizations and polynomial ring extensions.

In \cite{Z g-ded}, this author studied integral domains $R$ with the
property that $A_{v}$ is invertible for every nonzero ideal $A,$ and called
these domains "generalized Dedekind Domains" or G-Dedekind domains. Later
using a new form of the above mentioned $\ast $-property, Anderson and B.G.
Kang \cite{AK} published a much improved version of \cite{Z g-ded}. Calling
the G-Dedekind domains "Pseudo Dedekind" domains, they showed that $R$ is a
Pseudo Dedekind domain if and only if for all sets $\{a_{\alpha }\}_{\alpha
\in I},$ $\{b_{\beta }\}_{\beta \in J}$ $\subseteq K\backslash \{0\}$ we
have $(\cap (a_{\alpha }))(\cap (b_{\beta })=\cap (a_{\alpha }b_{\beta })$,
where $\alpha $ ranges over $I$ and $\beta $ over $J.$ (This was actually a
clever translation of the characteristic property of the G-Dedekind domains: 
$(AB)^{-1}=A^{-1}B^{-1}$ for all $A,B\in F(R)$ given in \cite{Z g-ded}.)
Call a set of elements $\{a_{\alpha }\}_{\alpha \in I}\subseteq K\backslash
\{0\}$ allowable if $\cap (a_{\alpha })\neq (0).$ In this article, we show,
among things already indicated, that for a given star operation $\star ,$ $%
A^{\star }$ is invertible for any ideal $A$ of $R$ if and only if for any
allowable set of elements $\{a_{\alpha }\}_{\alpha \in I}\subseteq
K\backslash \{0\}$ and for every nonzero ideal $A$ we have $A^{\star }(\cap
(a_{\alpha }))=\cap A^{\star }a_{\alpha }.$ As a consequence we show that $R$
is a Pseudo Dedekind/G-Dedekind domain if and only if for each nonzero ideal 
$A$ of $R$ and for each allowable set $\{a_{\alpha }\}_{\alpha \in
I}\subseteq K\backslash \{0\},$ we have $A_{v}(\cap (a_{\alpha }))=\cap
A_{v}(a_{\alpha }).$ We show that $R$ is a DCD if and only if for any
allowable set $\{a_{\alpha }\}_{\alpha \in I}\subseteq K\backslash \{0\}$
there is a set $\{b_{1},b_{2},...,b_{n}\}\subseteq K\backslash \{0\}$ such
that $\cap _{\alpha \in I}(a_{\alpha })=(b_{1},b_{2},...,b_{n})_{v}$. We
also show that $R$ is a G-Dedekind/Pseudo Dedekind domain if and only if $R$
is a DCD with the above mentioned property $\ast $. That is, the DCDs were
essentially at work behind the scenes in the results in \cite{Z g-ded} and
subsequently in \cite{AK}.

We use star operations, in order to approach the subject from a more general
standpoint. Included below is a brief introduction to star operations. The
reader may consult \cite{Gil}, \cite{E} or \cite{H-K} for more information
on star operations. For our purposes we include below some information that
may be helpful in reading this article.

Let $R$ be an integral domain with quotient field $K$ and let $F(R)$ be the
set of nonzero fractional ideals of $R.$ A star operation is a function $%
A\mapsto A^{\star }$ on $F(R)$ with the following properties:

If $A,B\in F(R)$ and $a\in K\backslash \{0\},$ then

(i) $(a)^{\star }=(a)$ and $(aA)^{\star }=aA^{\star }.$

(ii) $A\subseteq A^{\star }$ and if $A\subseteq B,$ then $A^{\star
}\subseteq B^{\star }.$

(iii) $(A^{\star })^{\star }=A^{\star }.$

We may call $A^{\star }$ the $\star $-image ( or $\star $-envelope ) of $A.$
An ideal $A$ is said to be a $\star $\textit{-ideal} if $A^{\star }=A.$ Thus 
$A^{\star }$ is a $\star $-ideal (by (iii)). Moreover (by (i)) every
principal fractional ideal, including $R=(1)$, is a $\star $- ideal for any
star operation $\star $.

For all $A,B\in F(R)$ and for each star operation $\star $, we can show that 
$(AB)^{\star }=(A^{\star }B)^{\star }=(A^{\star }B^{\star })^{\star }$.
These equations define what is called $\star $-multiplication \textit{( or }$%
\star $-product)\textit{. }Associated with each star operation $\star $ is a
star operation $\star _{f}$ defined by $A^{\star _{f}}=\bigcup \{J^{\star }|$
$0\neq J$ is a finitely generated subideal of $A\},$ for each $A\in F(D).$
We say that a star operation $\star $ is of finite type or of finite
character if $\star =\star _{f}.$

Define $A^{-1}=\{x\in K|xA\subseteq R\}.$ Thus $A^{-1}=\cap _{a\in
A\backslash \{0\}}(\frac{1}{a}).$ Also define $A_{v}=(A^{-1})^{-1}$ and $%
A_{t}=A_{v_{f}}=\bigcup \{J_{v}|$ $0\neq J$ is a finitely generated subideal
of $A\}.$ By the definition $A_{t}=A_{v}$ for each finitely generated
nonzero ideal of $R.$ The functions $A\mapsto A_{v}$ and $A\mapsto A_{t}$ on 
$F(R)$ are more familiar examples of star operations defined on an integral
domain. A fractional ideal $A\in F(R)$ is $\star $-invertible if $%
(AA^{-1})^{\star }=R.$ An invertible ideal is a $\star $-invertible $\star $%
-ideal for each $\star $-operation $\star $ and so is a $v$-ideal. A $v$%
-ideal is better known as a divisorial ideal and using the definition it can
be shown that $A_{v}=\cap _{\substack{ x\in K\backslash \{0\}  \\ A\subseteq
xR}}xR$. The identity function $d$ on $F(R)$, defined by $A\mapsto A$ is
another example of a star operation. Indeed a "$d$-invertible" ideal is the
usual invertible ideal. There are of course many more star operations that
can be defined on an integral domain $R$. But for any star operation $\star $
and for any $A\in F(R),$ $A^{\star }\subseteq A_{v}.$ Some other useful
relations are: For any $A\in F(R),$ $(A^{-1})^{\star }=$ $A^{-1}=(A^{\star
})^{-1}$ and so, $(A_{v})^{\star }=A_{v}=(A^{\star })_{v}.$ Using the
definition of the $t$-operation one can show that an ideal that is maximal
w.r.t. being a proper integral $t$-ideal is a prime ideal of $R$, each
nonzero ideal $A$ of $R$ with $A_{t}\neq R$ is contained in a maximal $t$%
-ideal of $R$ and $R=\cap R_{M},$ where $M$ ranges over maximal $t$-ideals
of $R.$ The set of maximal $t$-ideals of $R$ is denoted by $t$-$Max(R).$ For
more on $v$- and $t$-operations the reader may consult sections 32 and 34 of
Gilmer \cite{Gil}, or the other two books cited above. This is the barest
minimum of description to get us started, we shall expand on it when need
arises. Our terminology comes from \cite{Gil}. Of course we have called a
subset $\{a_{\alpha }\}_{\alpha \in I}\subseteq K\backslash \{0\}$ allowable 
$\cap (a_{\alpha })\neq (0)$ (a) to save on space and (b) because the
characteristic property of a G-Dedekind/pseudo Dedekind domain is that for
all $A,B\in F(D)$ we have $(AB)^{-1}=A^{-1}B^{-1}$ and this does not require
any set $\{a_{\alpha }\}_{\alpha \in I}\subseteq K\backslash \{0\}$ with $%
\cap (a_{\alpha })=(0).$ The reader may expect such purpose oriented
terminology in the sequel as well. In section \ref{S2} we settle a name for
G-/Pseudo Dedekind domains, indicating reasons why we should, and in section %
\ref{S3} we show that $R$ is a G-/Pseudo Dedekind domain if and only if $R$
is a DC $\ast $-domain. Finally in section \ref{S4} we touch on some related
questions and indicate how DCDs behave under some extensions such as
quotient ring formation or, in case of integrally closed DCDs, polynomial
ring formation.

\section{What's in a name? \label{S2}}

Popescu \cite{Pop} introduced the notion of a generalized Dedekind domain
via localizing systems. Nowadays, the following equivalent definition is
usually given: an integral domain is a generalized Dedekind domain if it is
a strongly discrete Pr\"{u}fer domain (i.e., $P\neq P^{2}$ for every prime
ideal $P$) and every (prime) ideal $I$ has $\surd I=\surd (a_{1},...,a_{n})$
for some $a_{1},...,a_{n}\in I$ (or equivalently, every principal ideal has
only finitely many minimal prime ideals). Unbeknownst to this author \cite%
{Pop} was already out. I personally do not think there is anything pseudo
about the G-Dedekind domains. On the other hand some serious studies related
to G-Dedekind prime rings, introduced by Evrim Akalan \cite{Ak}, are being
carried out. Hence the need for a name close to G-Dedekind domains. The aim
of this section is to fix a suitable name for the domains that are given two
different names, one not quite appropriate and the other overshadowed by a
previously adopted name. Of course as a result we end up with one more
notion a more suitable name for a so called $\pi $-domain.

The following lemma essentially comes from \cite{A}, yet we use it to give a
more general view of what started as G-Dedekind domains or Pseudo Dedekind
domains. Of course our statement is different and more streamlined.

\begin{lemma}
\bigskip \label{Lemma 0}Let $\star $ be a star operation defined on an
integral domain $R$ and let $A\in F(R).$ Then $A^{\star }$ is invertible if
and only if for any allowable set of elements $\{a_{\alpha }\}_{\alpha \in
I}\subseteq K\backslash \{0\}$, we have $(A^{\star }(\cap (a_{\alpha
}))=\cap A^{\star }a_{\alpha }.$
\end{lemma}

The proof of the above lemma has been used as part of the proof of the
following theorem.

\begin{theorem}
\label{Theorem A}Let $\star $ be a star operation defined on an integral
domain $R.$ Then $A^{\star }$ is invertible for every nonzero fractional
ideal $A$ of $R$ if and only if for any allowable set of elements $%
\{a_{\alpha }\}_{\alpha \in I}\subseteq K\backslash \{0\}$ and for every
nonzero ideal $A$ we have $(A^{\star }(\cap (a_{\alpha }))=\cap A^{\star
}a_{\alpha }.$
\end{theorem}

\begin{proof}
Let $A$ be any nonzero ideal of $R$ and suppose that for every allowable set 
$\{a_{\alpha }\}_{\alpha \in I}\subseteq K\backslash \{0\}$, we have $%
A^{\star }(\cap (a_{\alpha }))=\cap A^{\star }a_{\alpha }.$ Then $R\supseteq
A^{\star }A^{-1}$  $=$    $A^{\star }(\cap _{a_{\beta }\in A\backslash
\{0\}}(\frac{1}{a_{\beta }})$   $=\cap _{a_{\beta }\in A\backslash
\{0\}}A^{\star }(\frac{1}{a_{\beta }})$ by the condition. Since for each of $%
a_{\beta }\in A\backslash \{0\}$ we have $A^{\star }(\frac{1}{a_{\beta }}%
)\supseteq R$ and so $R\supseteq (A^{\star }A^{-1}=$  $A^{\star }(\cap
_{a_{\beta }\in A\backslash \{0\}}(\frac{1}{a_{\beta }}))=\cap _{a_{\beta
}\in A\backslash \{0\}}A^{\star }(\frac{1}{a_{\beta }})\supseteq R,$ showing
that $R=A^{\star }(\cap _{a_{\beta }\in A\backslash \{0\}}(\frac{1}{a_{\beta
}}))=A^{\star }A^{-1}.$ Thus, as $A$ was chosen arbitrarily, the condition
implies that for every nonzero ideal $A$ we have that $A^{\star }$ is
invertible. Conversely suppose that $A$ is a nonzero ideal such that $%
A^{\star }$ is invertible. Then, by an exercise on page 80 of \cite{Gil} we
have, for any allowable set $\{a_{\alpha }\}_{\alpha \in I}\subseteq
K\backslash \{0\},$ $A^{\star }(\cap (a_{\alpha }))=\cap A^{\star
}(a_{\alpha }).$ Thus for every nonzero ideal $A$ the ideal $A^{\star }$
being invertible implies that for every nonzero ideal $A,$ and for every
allowable set $\{a_{\alpha }\}_{\alpha \in I}\subseteq K\backslash \{0\},$
we have $A^{\star }(\cap (a_{\alpha }))=\cap A^{\star }a_{\alpha }.$
\end{proof}

\begin{corollary}
\label{Corollary B}For $\star =d$, $R$ is a Dedekind domain if and only if
for each nonzero ideal $A$ of $R$ and for each allowable set $\{a_{\alpha
}\}_{\alpha \in I}\subseteq K\backslash \{0\},$ we have $A(\cap (a_{\alpha
}))=\cap A(a_{\alpha }).$
\end{corollary}

\begin{proof}
Indeed it is well known that $R$ is a Dedekind domain if and only if every
nonzero ideal of $R$ is invertible and Theorem \ref{Theorem A} provides the,
general, necessary and sufficient conditions for every nonzero ideal to be $%
\star $-invertible, when $\star =d.$
\end{proof}

Let's recall that an ideal $I$ is $t$-invertible if $I$ is $v$-invertible
and $I^{-1}$ is a $v$-ideal of finite type, that an ideal $I$ that is
invertible, is $\star $-invertible for every star operation $\star $ \cite{Z
put} and that an integral domain $R$ is a Krull domain if and only if every
nonzero ideal of $R$ is $t$-invertible \cite{Jaf}. Now if $I_{t}$ is
invertible, then $I_{t}$ and hence $I$ is $t$-invertible. Thus if, for each
ideal $I,$ $I_{t}$ is invertible in a domain $R$ then, $R$ is at least a
Krull domain. According to Theorem 1.10 of \cite{Z g-ded} if $I_{t}$ is
invertible for every nonzero ideal of $R,$ then $R$ is a locally factorial
Krull domain. Such domains are often called $\pi $-domains for some reason.
Now Theorem \ref{Theorem A} characterizes $\pi $-domains for $\star =t$ as
follows.

\begin{corollary}
\label{Corollary C}A domain $R$ is a $\pi $-domain if and only if for each
nonzero ideal $A$ of $R$ and for each allowable set $\{a_{\alpha }\}_{\alpha
\in I}\subseteq K\backslash \{0\},$ we have $A_{t}(\cap (a_{\alpha }))=\cap
A_{t}(a_{\alpha }).$
\end{corollary}

\begin{proof}
Indeed $R$ is a $\pi $-domain if and only if $A_{t}$ is invertible for every
nonzero ideal $A$ of $R$ (Theorem 1.10 of \cite{Z g-ded}) and Theorem \ref%
{Theorem A} provides the, general, necessary and sufficient conditions for
every nonzero ideal to be $\star $-invertible, when $\star =t$.
\end{proof}

Of course by saying that an ideal $A$ of $R$ is $v$-invertible we mean that $%
(AA^{-1})_{v}=R.$ Similar to earlier comments we note that $A_{v}$ being
invertible entails $A$ being $v$-invertible. We also note that the domains $%
R $ with $A_{v}$ invertible for each nonzero ideal $A$ are the
G-Dedekind/Pseudo Dedekind domains. So for $\star =v,$ Theorem \ref{Theorem
A} provides the following characterization of G-Dedekind/Pseudo Dedekind
domains.

\begin{corollary}
\label{Corollary D}A domain $R$ is a G-Dedekind/Pseudo Dedekind domain if
and only if for each nonzero ideal $A$ of $R$ and for each allowable set $%
\{a_{\alpha }\}_{\alpha \in I}\subseteq K\backslash \{0\},$ we have $%
A_{v}(\cap (a_{\alpha }))=\cap A_{v}(a_{\alpha }).$
\end{corollary}

The proof is the same as the one provided by Theorem \ref{Theorem A}
replacing $\star $ by $v.$

\begin{remark}
\label{Remark D1} Looking at Theorem \ref{Theorem A} and Corollaries \ref%
{Corollary B}, \ref{Corollary C}, \ref{Corollary D}, we may call $R$ a $%
\star $-Dedekind domain if $A^{\star }$ is invertible for each $A\in F(R)$
and note that in a $\star $-Dedekind domain we have $A^{\star }=A_{v}$ for
all $A\in F(R).$ This is because an invertible ideal is divisorial. Thus
Corollary \ref{Corollary C} gives for $\star =t$ a $t$-Dedekind domain and
Corollary \ref{Corollary D} gives the name of a $v$-Dedekind domain to the
G-Dedekind domain of \cite{Z g-ded} and Pseudo Dedekind domain of \cite{AK}.
But there is a slight problem with this naming system, Jesse Elliott in \cite%
{E} calls a Krull domain a $t$-Dedekind domain. So, perhaps, $\star $%
-G-Dedekind may be the general name with the note that a $d$-G-Dedekind
domain is the usual Dedekind domain and a $t$-G-Dedekind domain is a locally
factorial Krull domain while the $v$-G-Dedekind domain is the usual Pseudo
Dedekind domain, or the old G-Dedekind domain. Of course, as $\star =v$ in a 
$\star $-G-Dedekind domain, each $\star $-G-Dedekind domain has the
properties listed in \cite{Z g-ded} for G-Dedekind domains are shared by $%
\star $-G-Dedekind domains, or in \cite{AK} for Pseudo Dedekind domains.
Thus if $\star $ is of finite character, then $\star =\star _{f}=v_{f}=t.$
That is if $R$ is a $\star $-G-Dedekind domain and $\star $ is of finite
character, then $R$ is a $t$-G-Dedekind domain. Recall that an integral
domain $R$ with quotient field $K$ is completely integrally closed if
whenever $rx^{n}\in R$ for $x\in K$, $0\neq r\in R$, and every integer $%
n\geq 1$, then $x\in R$. Equivalently, $R$ is completely integrally closed
if and only if $(AA^{-1})_{v}=R$ for every $A\in F(R)$ \cite{Gil}. If, for a
star operation $\star $ defined on $R,$ $A$ is $\star $ invertible for each $%
A\in F(R)$ following \cite{AAFZ} we may call $R$ a $\star $-CICD. Thus, as
noted in \cite{Z g-ded}, a $\star $-G-Dedekind domain is a completely
integrally closed domain (CICD), for each star operation $\star $ that it is
defined for. There is a star operation called the $w$-operation, defined in
terms of the $t$-operation as $A\mapsto A_{w}=\cap _{M\in t\text{-}%
Max(R)}AR_{M},$ see \cite{AC} and references there. As indicated in \cite{AC}%
, this operation is of finite character. Thus, in view of earlier comments
in this remark, a $w$-G-Dedekind domain is a $t$-G-Dedekind domain. (While
the $w$-operation has been around for some time, Wang and McCasland adopted
it in \cite{WM}.)
\end{remark}

\section{Dually compact domains \label{S3}}

Cohn \cite{C} called an element $x\in R$ primal if for all $y,z\in R,x|yz$
implies that $x=rs$ where $r|y$ and $s|z.$ A domain all of whose nonzero
elements are primal was called a pre-Schreier domain in \cite{Z pres}; this
was a break from Cohn who called $R$ a Schreier domain if $R$ was integrally
closed with all elements primal. Based on a study of the group of
divisibility of a pre-Schreier domain this author extracted, in \cite{Z pres}%
, what he called the $\ast $ property, saying: $R$ is a $\ast $ domain if
for all pairs of subsets $\{a_{1},...,a_{m}\},\{b_{1},...b_{n}\}\subseteq
K\backslash \{0\},$ $(\cap _{i=1}^{m}(a_{i})(\cap _{j=1}^{n}(b_{j}))=\cap
_{i,j=1}^{m,n}(a_{i}b_{j}).$

We now look at the facts working behind the Anderson-Kang/Zafrullah results.
For this let us call a domain $R$ dually compact (DC) if for any allowable
subset $\{a_{\alpha }\}_{\alpha \in I}\subseteq K\backslash \{0\},$ there is
a set $\{x_{1},x_{2},...,x_{n}\}\subseteq K\backslash \{0\}$ such that $\cap
_{\alpha \in I}(a_{\alpha })=(x_{1})\cap (x_{2})\cap ...\cap (x_{n}).$ Let's
also note that a fractional ideal $A$ being divisorial ideal (or $v$-ideal)
of finite type means that there are elements $s_{1}...,s_{r}\in K\backslash
\{0\}$ such that $A=(s_{1}...,s_{r})_{v}.$

\begin{theorem}
\label{Theorem E}The following are equivalent for an integral domain $R$
that is different from its quotient field $K.$
\end{theorem}

(1) $R$ is DC,

(2) $A_{v}$ is a $v$-ideal of finite type for every $A\in F(R),$

(3) $A^{-1}$ is of finite type for each $A\in F(R),$

(4) $A^{-1}$ is of finite type for each nonzero integral ideal $A$ of $R,$

(5) $A_{v}$ is of finite type for each nonzero integral ideal $A$ of $R,$

(6) every divisorial ideal of $R$ is expressible as a finite intersection of
principal fractional ideals,

(7) for a star operation $\star $ and for each $A\in F(R),$ the ideal $%
A^{\star }$ is a $v$-ideal of finite type,

(8) for a star operation $\star $ and for each $A\in F(R),$ the ideal $%
A^{\star }$ is a finite intersection of principal ideals from $F(R).$

\begin{proof}
(1) $\Rightarrow $ (2). Suppose that $R$ is DC and suppose that $A$ is a
nonzero fractional ideal generated by $\{c_{\gamma }\}_{\gamma \in J}.$ Then 
$A^{-1}=\cap _{\gamma }(\frac{1}{c_{\gamma }})\neq (0).$ So by the DC
property there is a finite set $\{x_{1},x_{2},...,x_{n}\}\subseteq
K\backslash \{0\}$ such that $\cap _{\gamma }(\frac{1}{c_{\gamma }})=\cap
_{i=1}^{n}(x_{i})=(x_{1}^{-1},...,x_{n}^{-1})^{-1}.$ Now $%
A^{-1}=(x_{1}^{-1},...,x_{n}^{-1})^{-1}$ implies that $%
A_{v}=(x_{1}^{-1},...,x_{n}^{-1})_{v}.$

(2) $\Rightarrow $ (1). Suppose that for each $A\in F(R),$ there are $%
x_{1},x_{2},...,x_{n}\in K\backslash \{0\}$ such that $%
A_{v}=(x_{1},x_{2},...,x_{n})_{v}$. Then since $A^{-1}$ is a divisorial
ideal, as $(A^{-1})_{v}=A^{-1},$ we have $A^{-1}=(y_{1},...,y_{r})_{v}.$ Or,
assuming that all the $y_{i}$ are nonzero, $A_{v}=\cap (\frac{1}{y_{i}})$
for each $A\in F(R).$ Thus if for each $A\in F(R)$ $A_{v}$ is of finite
type, then for each $A\in F(R)$ we can find some $b_{i}\in K\backslash \{0\}$
such that $A_{v}=\cap _{i=1}^{r}(b_{i}).$ Now let $\{a_{\alpha }\}_{\alpha
\in I}\subseteq K\backslash \{0\}$ be allowable and let $A=\cap (a_{\alpha
}).$ Since $\cap (a_{\alpha })\neq (0),$ $A$ is a divisorial ideal by \cite%
{Gil} and hence of finite type and so, by (2) for some $x_{1},...,x_{n}\in
K\backslash \{0\}$ we have $A=\cap (a_{\alpha })=\cap _{i=1}^{n}(x_{i}).$

Next (2) $\Leftrightarrow $ (5) and (3) $\Leftrightarrow $ (4) because every
fractional ideal $A$ of $R$ is of the form $\frac{B}{d}$ where $B$ is an
integral ideal. (2) $\Rightarrow $ (3) because $A^{-1}=(A^{-1})_{v}$ and (3) 
$\Rightarrow $ (2) because $A_{v}=(A^{-1})^{-1}.$

(1) $\Rightarrow $ (6). A nonzero ideal $A$ is divisorial if and only if $%
A=\cap _{\substack{ A\subseteq x_{\alpha }R  \\ x\in K\backslash \{0\}}}%
x_{\alpha }R.$ By the DC condition there are $x_{1}.,,,x_{n}$ in $%
K\backslash \{0\}$ such that $A=\cap _{\substack{ A\subseteq x_{\alpha }R 
\\ x\in K\backslash \{0\}}}x_{\alpha }R=\cap _{i=1}^{n}x_{i}R,$

(6) $\Rightarrow $ (1). Let for $\{a_{\alpha }\}_{\alpha \in I}\subseteq
K\backslash \{0\},$ $\cap (a_{\alpha })\neq (0)$. Note that $\cap (a_{\alpha
})$ is divisorial. So by (6) there are elements $x_{1}...,x_{r}\in
K\backslash \{0\}$ such that $\cap (a_{\alpha })=\cap _{i=1}^{r}(x_{i}).$
Finally each of (7) and (8) holds if and only if $A^{\star }=A_{v}$ for all $%
A\in F(R)$ because in both cases $A^{\star }=B$ where $B$ is divisorial and
so $A_{v}=B=A^{\star }.$ Observing that, we have (7) $\Leftrightarrow $ (5)
and (8) $\Leftrightarrow $ (6).
\end{proof}

\begin{remark}
\label{Remark E1} (1) There are a number of integral domains that fit the
description of DCDs. Noetherian domains do fit nicely, as do the so-called
Mori domains. Recall that $R$ is a Mori domain if $R$ satisfies ascending
chain conditions on divisorial ideals. It is well known that $R$ is a Mori
domain if and only if for each $A\in F(R)$ there is a finitely generated
fractional ideal $B\subseteq A$ such that $A_{v}=B_{v}.$ Indeed a Krull
domain is a DCD, being a Mori domain as indicated in Fossum \cite{F}. On the
other hand, there are DCDs, such as the ring of entire functions in which $%
A_{v}$ is principal for each $A\in F(R)$ and, the ring of entire functions
is neither a Mori domain, nor a Krull domain. (2) If the star operation $%
\star $ defined on $R$ is such that $A^{\star }$ is a $v$-ideal of finite
type for each $A\in F(R)$ we can call $R$ a $\star $-DCD. (3) It may be
somewhat hard to see, for some, that any $t$-invertible $t$-ideal, and hence
any invertible ideal, is a finite intersection of principal fractional
ideals. Let me note for the record that if $A$ is a $t$-invertible $t$%
-ideal, then $A^{-1}$ is of finite type. Say for some $B=\{b_{1},...b_{n}\}$
we have $A^{-1}=B_{v}.$ But then $A=A_{v}=(B_{v})^{-1}=\cap _{i=1}^{n}(\frac{%
1}{b_{i}}).$
\end{remark}

Indeed as in a DCD the inverse of every nonzero fractional ideal is of
finite type, all we need for the $v$-G-Dedekind property to hold is the $%
\ast $-property.

\begin{theorem}
\label{Theorem F}Let $R$ be a DCD. Then $R$ is a $v$-G-Dedekind domain if
and only if $R$ is a $\ast $-domain.
\end{theorem}

\begin{proof}
It is easy to see, from the treatment of it in \cite{AK}, that a $v$%
-G-Dedekind domain is a $\ast $-domain. (This fact was also mentioned in 
\cite{Z g-ded}.) But of course we need to show that a $v$-G-Dedekind domain
is DC. This follows from the fact that, in a $v$-G-Dedekind domain $R,$ $%
A_{v}$ is invertible for each $A\in F(R)$ and, as shown in Remark \ref%
{Remark E1}, an intersection of finitely many principal fractional ideals.
For the converse we show that DC plus the $\ast $-property implies the $v$%
-G-Dedekind property. For this consider for allowable sets $\{a_{\alpha
}\}_{\alpha \in I}\subseteq K\backslash \{0\},$ $\{b_{\beta }\}_{\beta \in
J}\subseteq K\backslash \{0\}$ the product $P=(\cap (a_{\alpha }))(\cap
(b_{\beta })).$ By the DC property of $R$ we can find $\{x_{1},...,x_{m}\},%
\{y_{1},...y_{n}\}\subseteq K\backslash \{0\}$ such that $\cap (a_{\alpha
})=\cap _{i=1}^{m}(x_{i})$ and $\cap (b_{\beta })=(\cap _{j=1}^{n}(y_{j}).$
Thus using DC plus $\ast ,$ $P=(\cap (a_{\alpha }))(\cap (b_{\beta }))=(\cap
(x_{i}))(\cap (y_{j}))=\cap _{i,j=1}^{m,n}(x_{i}y_{j})=\cap _{i}x_{i}(\cap
(y_{j}))=\cap x_{i}(\cap (b_{\beta }))$ (using $\cap (y_{j})=\cap (b_{\beta
})).$ This gives $P=\cap _{i}x_{i}(\cap (b_{\beta }))=\cap _{i,\beta
}(x_{i}b_{\beta })=\cap _{\beta }b_{\beta }(\cap (x_{i}))=\cap _{\beta
}b_{\beta }(\cap (a_{\alpha }))=\cap _{\alpha ,\beta }(a_{\alpha }b_{\beta
}).$
\end{proof}

Recall that an integral domain $R$ is called $v$--coherent if every finite
intersection of $v$--ideals of finite type is a $v$--ideal of finite type,
equivalently, if $A^{-1}$ is a $v$--ideal of finite for all finitely
generated $A\in F(R),$ \cite{FG}. Also that $R$ is a GGCD domain if $aR\cap
bR$ is invertible for all $a,b\in R\backslash \{0\},$ \cite{AA}. According
to Corollary 1.7 of \cite{Z g-ded} a $v$-coherent domain $R$ is a GGCD
domain if and only if $R$ is a $\ast $-domain. Now a DCD $R$ is slightly
more than a $v$-coherent domain.

\begin{corollary}
\label{Corollary F1}Suppose that for a star operation $\star $, $A^{\star }$
is a $v$-ideal of finite type for each $A\in F(R)$. Then $R$ is a $\star $%
-G-Dedekind domain if and only if $R$ is a $\ast $-domain. Consequently a
Mori (and hence a Krull) domain $R$ is a locally factorial Krull domain if
and only if $R$ is a $\ast $-domain.
\end{corollary}

The proof can be easily constructed from the preceding comments and so is
left to the reader.

Considered in \cite{AK} was also the notion of a pseudo principal ideal
domain (pseudo PID) or, in our terminology, $v$-G-PID by requiring that $%
A_{v}$ is principal, for each nonzero ideal $A$ of $R.$ This notion appeared
in Bourbaki too, as pointed out in \cite{AK}. Using the DC approach, one can
prove the following result.

\begin{theorem}
\label{Theorem G}A domain $R$ is a $v$-G-PID if and only if $R$ is a DC
Schreier domain.
\end{theorem}

\begin{proof}
Suppose that $R$ is a DC Schreier domain. Now a Schreier domain is a $\ast $%
-domain too \cite[Corollary 1.7]{Z pres}. So a DCD that is a Schreier domain
is at least a $v$-G-Dedekind domain, by Theorem \ref{Theorem F}. But then $%
A_{v}$ is invertible for each nonzero ideal $A$ of $R$ and in a Schreier
domain every invertible ideal is principal \cite[Theorem 3.6]{Z pres}. Thus $%
A_{v}$ is principal for each nonzero ideal $A$ of $R$ and $R$ is a $v$%
-G-PID. Conversely note that a $v$-G-PID is DC and is at least a GCD domain
and a GCD domain is Schreier \cite{C}. Thus a $v$-G-PID $R$ is a DC Schreier
domain.
\end{proof}

\begin{theorem}
\label{Theorem H}Let $\star $ be a star operation defined on $R$ such that
for each $A\in F(R),$ $A^{\star }$ is a $v$-ideal of finite type. Then the
following are equivalent.
\end{theorem}

(1) $R$ is a $\star $-G-PID,

(2) $R$ is a Schreier domain,

(3) $R$ is a $\star $-G-Dedekind domain with $Pic(R)=(0).$

\begin{proof}
(1) $\Leftrightarrow $ (2). A $\star $-G-PID, by definition, is a GCD domain
and so a Schreier domain. Conversely, "for each $A\in F(R),$ $A^{\star }$ is
a $v$-ideal of finite type" makes each $A^{\star }\in F(R)$ a $v$-ideal of
finite type. But then for each $A\in F(D)$ $A^{\star }$ is a finite
intersection of principal fractional ideals and of finite type and by
Theorem 3.6 of \cite{Z pres} principal because $R$ is Schreier.

(1) $\Leftrightarrow $ (3). $\star $-GPID is $\star $-G-Dedekind with every
invertible ideal principal which is exactly (3).

For the converse note that (3) implies that $R$ is at least a GCD domain and
a GCD domain is Schreier. That is (2) holds and (2) is equivalent to (1).
\end{proof}

\section{Related stuff \label{S4}}

We end this article with some interesting characterizations of the $\star $%
-GPIDs and discussion of related material.

Recall that a Riesz group is a directed group that satisfies the Riesz
interpolation property: given that $%
x_{1},x_{2},...,x_{m};y_{1},y_{2},...,y_{n}\in G$ such that $x_{i}\leq y_{j}$
for all $i\in \lbrack 1,m],j\in \lbrack 1,n]$ there is $z\in G$ such that $%
x_{i}\leq z\leq y_{j}$ for all $(i,j)\in \lbrack 1,m]\times \lbrack 1,n]).$
It was shown in \cite{Z pres} that the Riesz interpolation property
translates in the commutative ring theory set up to: for all $%
x_{1},x_{2},...,x_{m};y_{1},y_{2},...,y_{n}\in K\backslash \{0\}$ with $%
x_{1},x_{2},...,x_{m}\in \cap y_{i}R$ there is a $z\in \cap y_{i}R$ such
that $(x_{1},...,x_{m})\subseteq zR\subseteq \cap y_{i}R.$ So a pre-Schreier
domain is actually a pre-Riesz domain. We have no interest in changing
existing names, we only want to add a new name. Call $R$ a super Riesz
domain if for any divisorial ideal $A$ of $R$ and for any set $\{x_{\alpha
}\}$ of elements contained in $A,$ with $\cap (x_{\alpha })\neq (0),$ there
is a $d\in A$ such that $(x_{\alpha })\subseteq (d).$ Because a divisorial
ideal is expressible as an intersection of principal fractional ideals, a
super Riesz domain can be easily seen to be pre-Schreier. Also let's call a
product $AB$ of ideals $A,B$ subtle if for each $x\in AB$ we have $x=ab$
where $a\in A$ and $b\in B.$ The authors of \cite{AK} also touched on the
following question: Let $R$ be an integral domain that satisfies $(\cap
(a_{\alpha }))(\cap (b_{\beta })=\cap (a_{\alpha }b_{\beta })$ for all
subsets $\{a_{\alpha }\}_{\alpha \in I}\subseteq R\backslash \{0\},$ $%
\{b_{\beta }\}_{\beta \in J}\subseteq R\backslash \{0\}$. Is $R$
pseudo-Dedekind?

We try to give a partial answer below.

\begin{proposition}
\label{Proposition K}Let $R$ be an integral domain. Then the following are
equivalent.
\end{proposition}

(1) For all allowable $\{a_{\alpha }\}_{\alpha \in I}\subseteq K\backslash
\{0\},$ $\{b_{\beta }\}_{\beta \in J}\subseteq K\backslash \{0\}$ we have $%
(\cap (a_{\alpha }))(\cap (b_{\beta })=\cap (a_{\alpha }b_{\beta }),$

(2) for all allowable $\{a_{\alpha }\}_{\alpha \in I}\subseteq R\backslash
\{0\},$ $\{b_{\beta }\}_{\beta \in J}\subseteq R\backslash \{0\}$ we have $%
(\cap (a_{\alpha }))(\cap (b_{\beta })=\cap (a_{\alpha }b_{\beta })$ and for
all allowable sets $\{a_{\alpha }\}_{\alpha \in I}\subseteq K\backslash
\{0\} $ we have a $d\in R\backslash \{0\}$ such that $d(\cap (a_{\alpha }))$
is expressible as an intersection of principal integral ideals.

(3) $R$ is a DCD that satisfies: for all allowable $\{a_{\alpha }\}_{\alpha
\in I}\subseteq R\backslash \{0\},$ $\{b_{\beta }\}_{\beta \in J}\subseteq
R\backslash \{0\}$ we have $(\cap (a_{\alpha }))(\cap (b_{\beta }))=\cap
(a_{\alpha }b_{\beta }).$

\begin{proof}
(1) $\Rightarrow $ (2). Obvious because the first part follows directly and
(1) means $R$ is DC and so for each allowable set $\{a_{\alpha }\}_{\alpha
\in I}\subseteq K\backslash \{0\}$ we can find $x_{1},...,x_{n}$ such that $%
(\cap (a_{\alpha }))=(\cap _{i=1}^{n}(x_{i}))$. But for the right hand we
can find a nonzero $d$ such that $dx_{i}$ are all integral.

(2) $\Rightarrow $ (1). Consider $(\cap (a_{\alpha }))(\cap (b_{\beta })$
for all allowable $\{a_{\alpha }\}_{\alpha \in I}\subseteq K\backslash
\{0\}, $ $\{b_{\beta }\}_{\beta \in J}\subseteq K\backslash \{0\}.$ Since $%
(\cap (a_{\alpha })),(\cap (b_{\beta }))$ are nonzero fractional ideals by
the given property of $R$ we have, for some $r,s\in R,$ $r(\cap (a_{\alpha
}))=\cap (x_{\gamma })$ and $s(\cap (b_{\beta }))=$ $\cap (y_{\delta })$,
where $x_{\gamma }$ and $y_{\delta }$ are in $R,$ by (2). But then $(\cap
(x_{\gamma }))(\cap (y_{\delta }))=\cap (x_{\gamma }y_{\delta })=P.$ Now $%
P=\cap (x_{\gamma }y_{\delta })$ $=\cap _{\gamma }x_{\gamma }(\cap
(y_{\delta }))$ (substituting for $\cap (y_{\delta })$) $=\cap _{\gamma
}x_{\gamma }(\cap (sb_{\beta }))=\cap _{\gamma ,\beta }(x_{\gamma }sb_{\beta
})=\cap _{\beta }sb_{\beta }(\cap (x_{\gamma }))$ and substituting for $\cap
(x_{\gamma })$) we get $P=\cap _{\beta }sb_{\beta }(\cap (ra_{\alpha
}))=\cap (ra_{\alpha }sb_{\beta })=rs(\cap (a_{\alpha }b_{\beta })).$ But on
the other hand $P=(\cap (x_{\gamma }))(\cap (y_{\delta }))=(r(\cap
(a_{\alpha })))(s(\cap (b_{\beta }))).$ Thus $rs(\cap (a_{\alpha }))(\cap
(b_{\beta }))=rs(\cap (a_{\alpha }b_{\beta })).$ Canceling $rs$ from both
sides we get the desired equality.

(1) $\Rightarrow $ (3). Obvious, in light of (1) $\Rightarrow $ (2).

(3) $\Rightarrow $ (2). Note that because $R$ is DC for each allowable set $%
\{a_{\alpha }\}_{\alpha \in I}\subseteq K\backslash \{0\}$ we can find $%
x_{1},...,x_{n}$ such that $(\cap (a_{\alpha }))=(\cap _{i=1}^{n}(x_{i}))$
and so a $d\in R\backslash \{0\}$ such that $dx_{i}\in R.$ But then $d(\cap
(a_{\alpha }))=\cap _{i-1}^{n}dx_{i}$ an intersection of principal
fractional ideals.
\end{proof}

Recall from \cite{Z pres}, again, that $R$ is a pre-Schreier domain if and
only if for all $\{a_{1},...,a_{m}\},\{b_{1},...,b_{n}\}\subseteq
R\backslash \{0\},a_{i}b_{j}|x$ implies $x=rs$ where $a_{i}|r$ and $b_{j}|s.$
Since this property sprang in the context of pre-Schreier domains we can
call a domain $R$ super pre-Schreier if $\{a_{\alpha }\}_{\alpha \in
I}\subseteq R\backslash \{0\},$ $\{b_{\beta }\}_{\beta \in J}\subseteq
R\backslash \{0\}$ such that $a_{\alpha }b_{\beta }|x$ then $x=rs$ such that 
$a_{\alpha }|r$ and $b_{\beta }|s,$ and ask if a super pre-Schreier domain
must be a $v$-GPID. The reason for this question is provided by the
following proposition.

\begin{proposition}
\label{Proposition L} An integral domain $R$ is super pre-Schreier if and
only if for all allowable $\{a_{\alpha }\}_{\alpha \in I}\subseteq
R\backslash \{0\},$ $\{b_{\beta }\}_{\beta \in J}\subseteq R\backslash \{0\}$
we have $(\cap (a_{\alpha }))(\cap (b_{\beta }))=\cap (a_{\alpha }b_{\beta
}) $ such that for each $x\in (\cap (a_{\alpha }))(\cap (b_{\beta }))$ $x=rs$
where $r\in \cap (a_{\alpha })$ and $s\in \cap (b_{\beta }).$
\end{proposition}

\begin{proof}
Suppose that $R$ is super pre-Schreier. Then $(\cap (a_{\alpha }))(\cap
(b_{\beta }))\subseteq (\cap (a_{\alpha }b_{\beta }))$ holds, always. So let 
$x\in \cap (a_{\alpha }b_{\beta }).$ This means $a_{\alpha }b_{\beta }|x.$
But super pre-Schreier property requires that $x=rs$ where $r\in \cap
(a_{\alpha })$ and $s\in \cap (b_{\beta })$ and that puts $x\in (\cap
(a_{\alpha }))(\cap (b_{\beta }))$ and the other requirement is met. The
converse is a similar translation.
\end{proof}

Call the product $IJ$ of two nonzero integral ideals $I,J$ of $R$ subtle if
each $d\in IJ\backslash \{0\}$ can be written as $d=rs$ where $r\in I$ and $%
s\in J.$ It was shown in \cite{Z pres} that $R$ is pre-Schreier if and only
if $R$ has the property $\ast $ and for each pair of subsets $%
\{a_{1},...,a_{m}\},$ $\{b_{1},...,b_{n}\}\subseteq R\backslash \{0\}$ the
product $(\cap _{i=1}^{n}(a_{i}))\left( \cap _{i=1}^{n}(a_{i})\right) $ is
subtle. (Following an earlier version of \cite{Z pres}, Anderson and Dobbs 
\cite{AD} studied domains products of whose ideals were all subtle.)

\begin{corollary}
\label{Corollary M} The following are equivalent for an integral domain $R.$
(1) For all allowable $\{a_{\alpha }\}_{\alpha \in I}\subseteq R\backslash
\{0\},$ $\{b_{\beta }\}_{\beta \in J}\subseteq R\backslash \{0\}$ we have $%
(\cap (a_{\alpha }))(\cap (b_{\beta }))=(\cap (a_{\alpha }b_{\beta }))$ such
that for each $x\in (\cap (a_{\alpha }))(\cap (b_{\beta }))$ $x=rs$ where $%
r\in (\cap (a_{\alpha }))$ and $s\in (\cap (b_{\beta })),$ (2) for every
pair of nonzero ideals $A,B$ we have $(AB)_{v}=A_{v}B_{v}$, for every
nonzero integral ideal $A$ of $R$ there is a $d\in R\backslash \{0\}$ with $%
dA_{v}$ an intersection of principal integral ideals and the product is
subtle and (3) $R$ is a super pre-Schreier domain.
\end{corollary}

Finally, a word about DCDs. Indeed a DCD can be characterized by: $A_{v}$ is
a $v$-ideal of finite type for each $A\in F(R).$ The first thing that comes
to mind as a general property is that $A^{-1}$ is of finite type for each $%
A\in F(R).$ This leads to the following result. But we need to recall some
terminology. $R$ is $\star $-Prufer if for each finitely generated $A\in
F(R) $ is $\star $-invertible \cite{AAFZ}. Since, for $A\in F(R),$ $A$ being 
$\star $-invertible implies $A^{\star }=A_{v}$. So for each finitely
generated nonzero ideal $J$ in a $\star $-Prufer domain $R$ we have $%
J^{\star }=J_{v}.$ Also as $(JJ^{-1})^{\star }=R$ implies $((JJ^{-1})^{\star
})_{v}=(JJ^{-1})_{v}=R$, a $\star $-Prufer domain is a $v$-domain and $\star
_{f}=t.$ If $\star $ is of finite character, then $A^{-1}$ is a $v$-ideal of
finite type for each finitely generated $A\in F(R)$ and so, if $\star $ is
of finite type, a $\star $-Prufer domain is a Prufer $\star $-multiplication
domain and $\star =t.$ $\star $-Prufer domains, for a finite character $%
\star $-operation $\star ,$ were studied in \cite{HMM} where they were
called $\star $-multiplication domains. They were later called Prufer $\star 
$-multiplication domains (P$\star $MDs). An in-depth study of these domains
can be found in \cite{FJS}, along with an introduction to semistar
operations. These days this concept is defined by: $R$ is a P$\star $MDs,
for a star operation $\star ,$ if for each finitely generated $A\in F(R)$, $%
A $ is $\star _{f}$-invertible. If $R$ is a P$\star $MD for a finite
character star operation, then $\star =t$ over $R.$ On the other hand if $R$
is a P$\star $MD for a "general" star-operation $\star ,$ then $\star
_{f}=t. $

\begin{proposition}
\label{Proposition N} Let $R$ be a DCD and let $\star $ be a star operation
defined on $R.$ If $A$ is $\star $-invertible, for $A\in F(R),$ then $%
A^{\star }$ is $\star _{f}$-invertible. Consequently a DCD $R$ is a P$\star $%
MD if and only if $R$ is a $\star $-Prufer domain.
\end{proposition}

\begin{proof}
Note that $A$ being $\star _{f}$-invertible means that $A^{-1}$ is a $v$%
-ideal of finite type. But $A^{-1}$ is a $v$-ideal of finite type in a DCD.
Thus a DC $\star $-Prufer domain is a P$\star $MD. The converse is obvious
in the case of a DCD.
\end{proof}

\begin{remark}
\label{Remark P} It may be noted that there is a marked difference between "$%
A$ is $\star _{f}$-invertible" and "$A^{\star }$ is $\star _{f}$%
-invertible". That is if you require that every $A\in F(R)$ is $\star _{f}$%
-invertible, you will end up with a Krull domain, for $\star _{f}$%
-invertible is $t$-invertible and every $A\in F(R)$ being $t$-invertible
implies that $R$ is Krull \cite[page 82]{Jaf}. On the other hand if you
require that for each $A\in F(R)$ the ideal $A^{\star }$ is $\star _{f}$%
-invertible, you get a PVMD that acts and behaves very much like $\star $%
-G-Dedekind domains. These domains were studied under the name of pre-Krull
domains in \cite{Z asc}, just for $\star =v$ and later, under the name of $%
(t,v)$-Dedekind domains, in \cite{AAFZ}, essentially based on the line taken
in \cite{Z asc}, describing the $(t,v)$-Dedekind domains as domains in which 
$(AB)^{-1}=(A^{-1}B^{-1})^{t},$ or as domains in which $A_{v}$ is $t$%
-invertible for each $A\in F(R)$. The closest to that \cite{AAFZ} came to
was in its Theorem 1.14. These and similar concepts were also studied by
Halter-Koch under "mixed invertibility". Hopefully, with the introduction of
DC property the situation will become clearer.
\end{remark}

For now we have the following corollary.

\begin{corollary}
\label{Corollary Q} The following are equivalent for an integral domain $R,$
with a star operation $\star $ defined on it.
\end{corollary}

(1) $A^{\star }$ is $\star _{f}$-invertible for each $A\in F(R),$

(2) In $R$ we have, $A$ $\star $-invertible and $(AB)^{\star }=(A^{\star
}B^{\star })^{\star _{f}},$ for all $A,B\in F(R),$

(3) $R$ is a $\star $-DC $\star $-Prufer domain.

(4) $R$ is a $\star $-CICD and $(AB)^{\star }=(A^{\star }B^{\star })^{\star
_{f}},$ for all $A,B\in F(R).$

\begin{proof}
(1) $\Rightarrow $ (2). Let $A^{\star }$ be $\star _{f}$-invertible, for all 
$A\in F(R).$ Then $A^{\star }$ and hence $A$ is $\star $-invertible, forcing 
$A^{\star }=A_{v}$ \cite{Z put}, for all $A\in F(R).$ Thus $\star =v$ over $%
R $. Moreover, as $\star _{f}$ is of finite character, $A^{\star }$ is a $%
\star $-ideal of finite type, because $A^{\star }$ is $\star _{f}$-invertible%
$.$ Now consider for $A,B\in R$ the product $(AB)^{\star }$. Indeed $%
(AB)^{\star }=(A^{\star }B^{\star })^{\star }$. Because each of $A^{\star },$
$B^{\star }$ is of finite type $(AB)^{\star }=(A^{\star }B^{\star })^{\star
_{f}}.$

(2) $\Rightarrow $ (1). $R=(AA^{-1})^{\star }=(A^{\star }A^{-1})^{\star
_{f}},$ by (2).

(2) $\Rightarrow $ (3). By (2) $A$ is $\star $-invertible for every $A\in
F(R)$ we conclude that $R$ is a $\star $-CICD and hence $\star $-Prufer.
Since by (2) we have $(AB)^{\star }=(A^{\star }B^{\star })^{\star _{f}}$ we
conclude that $R=(AA^{-1})^{\star }=(A^{\star }A^{-1})^{\star _{f}}$ which
established that $A^{\star }$ is a $\star $-ideal of finite type. But as $%
\star =v$ we conclude that $R$ is $\star $-DC.

(3) $\Rightarrow $ (1). Follows from Proposition \ref{Proposition N}.

Finally, (4) and (2) are re-statements of each other.
\end{proof}

The domain characterized by Corollary \ref{Corollary Q} is a completely
integrally closed PVMD, that is a Dedekind domain for $\star =d$ a Krull
domain for $\star $ of finite character, or for $\star =t,$ and a pre-Krull
domain of \cite{Z asc} for a star operation $\star $ that is not of finite
type, or $\star =v$. The following proposition is a reason why I tend to
call these domains $(\star ,v)$-G-Dedekind domains. For this let us recall
that $R$ is $\star $-DC if for each $A\in F(R)$ we have that $A^{\star }$ is
a $v$-ideal of finite type.

\begin{proposition}
\label{Proposition R} A $(\star ,v)$-G-Dedekind domain $R$ is a $\star $%
-G-Dedekind domain if and only if $R$ is a $\ast $-domain.
\end{proposition}

\begin{proof}
A $(\star ,v)$-G-Dedekind domain $R$ is $\star $-DC, with $\star =v,$ and so
DC while a DC $\ast $-domain is a indeed a $\star $-G-Dedekind domain, by
Theorem \ref{Theorem F}. Conversely a $\star $-G-Dedekind domain is a $\ast $%
-domain and DC. This means that it is $\star $-DC and as every $\star $%
-ideal in a $\star $-G-Dedekind domain is invertible we conclude that a $%
\star $-G-Dedekind domain is a $\ast $-domain that is a $(\star ,v)$%
-G-Dedekind domain for every finite type star operation $\star .$
\end{proof}

It may be noted that while $(\star ,v)$-G-Dedekind domains have been studied
before, Corollary \ref{Corollary Q} and Proposition \ref{Proposition R}
provide a better view and highlight the connection of $v$-G-Dedekind domains
with their generalization, the $(\star ,v)$-G-Dedekind domains. Now a look
at some general properties of DCDs.

It was indicated in \cite{Z g-ded}, using the example of the ring of entire
functions, a $v$-G-Dedekind domain does not behave well under localizations.
Using the same example we can conclude that a DCD does not behave well under
localization. But let us be clear about that example. Let $\mathcal{E}$
denote the ring of entire functions. Using the fact that $\mathcal{E}$ is a
Bezout domain and that an element $s$ of $\mathcal{E}$ that is not divisible
by any principal (height one) prime (i.e. one that does not have a zero) is
a unit, one can show that $\mathcal{E=\cap E}_{(p)}$ where $p$ ranges over
principal height one primes. In fact we have the following lemma.

\begin{lemma}
\label{Lemma R1} Let $R$ be a GCD domain such that an element $s\in
R\backslash \{0\}$ is a non unit if and only if $s$ is divisible by at least
one principal prime of $R.$ Then $R=\cap R_{(p)}$ where $p$ ranges over
principal height one primes of $R.$ Such a ring is completely integrally
closed.
\end{lemma}

\begin{proof}
Obviously if $R$ has no principal height one primes then $R$ is a field and $%
R$ is trivially an intersection of localizations at its height one primes.
So let us assume that $R$ does have a set $\wp $ of principal height one
primes and set $S=\cap _{p\in \wp }R_{(p)}.$ We already have $R\subseteq
\cap R_{(p)}.$ If there is $x\in S\backslash R,$ then $x=r/s$ and we can
assume that $GCD(r,s)=1.$ Now as $pR_{(p)}$ is of height one and $R$ a GCD
domain $R_{(p)}$ is a discrete rank one valuation domain. Next $r/s$ $\in
R_{(p)}$ for $p\in \wp $ forces $s$ to be a unit in $R_{(p)}$ meaning $s$ is
not divisible by $p.$ But then $x=r/s$ where $s$ is not divisible by any of
the $p\in \wp .$ This forces $s$ to be a unit, by the rule. But then $x$ is
an associate of $r\in R.$ Thus $R=\cap R_{(p)}.$ Finally, each $R_{(p)}$ for
each $p\in \wp $ is a discrete rank one valuation domain and hence
completely integrally closed and an intersection of completely integrally
closed domains is completely integrally closed.
\end{proof}

Thus we see the reason why one concludes that $\mathcal{E}$ is completely
integrally closed. Now $\mathcal{E}$ is Bezout and hence a PVMD such that
every nonzero ideal of $\mathcal{E}$ is a $t$-ideal. Following \cite{H div},
let $P$ be a nonzero non-maximal prime $t$-ideal of a PVMD $R$ and set $%
S=\cap R_{M}$ where $M$ ranges over all the maximal $t$-ideals which do not
contain $P.$ Then $P^{-1}=R_{P}\cap S$ \cite[Proposition 1.1]{H div}.

\begin{proposition}
\label{Proposition R2} Let $P$ be a nonzero prime of $\mathcal{E}$ of height
greater than one. Then $P_{v}=\mathcal{E}$.
\end{proposition}

\begin{proof}
Note that $P$ is contained in no height one prime of $\mathcal{E}$. But then
principal primes are maximal ($t$-) ideals in $\mathcal{E}$ and so by Lemma %
\ref{Lemma R1}, $S=\mathcal{E}$, forcing $P^{-1}=\mathcal{E}_{P}\cap 
\mathcal{E=E}$. Indeed then $P_{v}=\mathcal{E}$. Finally if $M$ is a maximal
ideal of $\mathcal{E}$ of height greater than one, then $M$ must contain a
nonzero non-maximal prime $\wp $ of height greater than one. But then $%
\mathcal{E=(\wp )}_{v}\subseteq M_{v}.$
\end{proof}

\begin{corollary}
\label{Corollary R3} In $\mathcal{E}$ it is possible to have a
multiplicative set $S$ and an ideal $A$ such that $(A\mathcal{E}%
_{S})_{v^{\prime }}\subsetneq (A_{v_{\mathcal{E}}}\mathcal{E}%
_{S})_{v^{\prime }}$ where $v^{\prime }$ and $v_{\mathcal{E}}$ are the $v$%
-operations in $\mathcal{E}_{S}$ and $\mathcal{E}$ respectively.
\end{corollary}

\begin{proof}
Let $P$ be non-zero non-maximal prime ideal of $\mathcal{E}$ of height
greater than one contained in a maximal ideal $M$ of $\mathcal{E}$. Then by
Proposition \ref{Proposition R2}, $P_{v_{\mathcal{E}}}=\mathcal{E}$, where $%
v_{\mathcal{E}}$ denotes the $v$-operation on $\mathcal{E}$. On the other
hand in $\mathcal{E}_{M}$ which is a valuation ring we have $P\mathcal{E}%
_{M} $ a non-maximal prime of $\mathcal{E}_{M}$ and so must be divisorial
being the intersection of all the principal ideals containing $P\mathcal{E}%
_{M}.$ Thus if $v^{\prime }$ is the $v$-operation on $\mathcal{E}_{M},$ then 
$(P\mathcal{E}_{M})_{v^{\prime }}=P\mathcal{E}_{M}\subsetneq \mathcal{E}%
_{M}=(P_{v_{\mathcal{E}}}E_{M})_{v^{\prime }}.$
\end{proof}

We now establish that it is the DC property going missing in localization
that is responsible for Corollary \ref{Corollary R3}. For this note that a
DCD that is also a $\ast $-domain is a $v$-G-Dedekind domain. Now the $\ast $%
-property, as indicated in \cite{Z pres}, fares nicely under localization
and so if the $v$-G-Dedekind property goes missing in localizing, it is the
DC property that goes missing in localizing. However, the situation can be
brought under control, once we introduce some restriction.

\begin{proposition}
\label{Proposition S} If $R$ is DC such that for each ideal $A\in F(R)$
there is a finitely generated $B\subseteq A$ with $A_{v}=B_{v}$, then for
each multiplicative set $S$ the ring $R_{S}$ is DC.
\end{proposition}

\begin{proof}
Note that a domain with the given property is DC because in it for every
ideal $A$ we have $A_{v}$ of finite type. On the other hand a domain with
the given property is known to be a Mori domain (\cite{Z asc} and references
there) and a Mori domain stays a Mori domain under localization. We include
the proof below.

Let $S$ be a multiplicative set in $R$ such that for each $A\in F(R)$ there
is a finitely generated $B\subseteq A$ with $A_{v}=B_{v}$ and let $\mathcal{%
\alpha }$ be an ideal of $R_{S}.$ Then $\alpha =AR_{S}$ where $A=\alpha \cap
R.$ Let $B\subseteq A$ where $B$ is finitely generated such that $%
A_{v}=B_{v}.$ Note that $A\subseteq B_{v}$ and so $\alpha =AR_{S}\subseteq
B_{v}R_{S}.$ Thus $\alpha _{v}=(AR_{S})_{v}\subseteq
(B_{v}R_{S})_{v}=(BR_{S})_{v},$ since $B$ is finitely generated \cite{Z fc}.
Since $B\subseteq A$, we have $BR_{S}\subseteq AR_{S}$ and so $%
(BR_{S})_{v}\subseteq (AR_{S})_{v}$. Thus $(BR_{S})_{v}=(AR_{S})_{v}$ with $%
BR_{S}$ finitely generated and contained in $AR_{S}.$
\end{proof}

Because in a DCD $R,$ $A_{v}$ is a finite intersection of principal
fractional ideals we conclude that $A_{v}R_{S}$ is divisorial. Thus $%
(AR_{S})_{v}\subseteq A_{v}R_{S}.$ However as Corollary \ref{Corollary R3}
indicates the inclusion may be strict on occasion. Now generally if $A\in
F(R)$ is finitely generated and $S$ a multiplicative set of $R$, then $%
(AR_{S})_{v^{\prime }}=(A_{v_{R}}R_{S})_{v^{\prime }},$ where $v^{\prime }$
and $v_{R}$ are $v$-operations on $R_{S}$ and $R$ respectively \cite{Z fc}.
This is a general formula and so it works for DCDs too, but in a slightly
modified form.

\begin{proposition}
\label{Proposition T} If $A\in F(R)$ is nonzero finitely generated, $S$ a
multiplicative set of $R$ and if $R$ is DC, then $(AR_{S})_{v^{\prime
}}=A_{v_{R}}R_{S},$ where $v^{\prime }$ and $v_{R}$ are $v$-operations on $%
R_{S}$ and $R$ respectively.
\end{proposition}

\begin{proof}
The proof follows from the fact that $A_{v_{R}}$ is a finite intersection of
principal fractional ideals and so $A_{v_{R}}R_{S}$ is a divisorial ideal of 
$R_{S}$ and that makes $A_{v_{R}}R_{S}=(A_{v_{R}}R_{S})_{v^{\prime }}.$
\end{proof}

A prime example of a DCD is a Mori domain and Roitman \cite{R mori} has
produced an example of a Mori domain $R$ such that $R[X]$ is not Mori. On
the other hand for integrally closed integral domains Querre \cite{Q} proved
the following result.

\begin{theorem}
\label{Theorem U} An integral domain $R$ is integrally closed if and only if
for every integral ideal $B$ of $R[X]$ with $B\cap R\neq 0$, we have $%
B_{v}=(A_{B}[X])_{v}=(A_{B})_{v}[X]$ where $A_{B}$ is the ideal of $R$
generated by the coefficients of elements of $B$.
\end{theorem}

For a clearer treatment of Theorem \ref{Theorem U}, see section 3 of \cite%
{AKZ}. For now, we use this theorem to prove the following result.

\begin{theorem}
\label{Theorem V} Let $R$ be an integrally closed integral domain. The
polynomial ring $R[X]$ is DC if and only if $R$ is.
\end{theorem}

\begin{proof}
Suppose that $R$ is DC. Then for every ideal $I\in F(R)$ of $R$ we have $%
I_{v}=J_{v}$ where $J$ is finitely generated. Now, let $H\in F(R[X]).$
Because $R$ is integrally closed, according to Theorem 2.1 of \cite{AKZ}, $H=%
\frac{f(X)}{g(X)}B$ where $f(X),g(X)\in R[X]$ and $B$ is an ideal of $R[X]$
with $B\cap R\neq (0).$ But then $H_{v}=\frac{f(X)}{g(X)}B_{v}=\frac{f(X)}{%
g(X)}((A_{B})_{v}[X])$ is a $v$-ideal of finite type because, as $A_{B}$ is
an integral ideal of $R,$ $(A_{B})_{v}$ is a $v$-ideal of finite type.
Conversely let $I$ be a nonzero integral ideal of $R$ and suppose that $R[X]$
is DC. Then $(I[X])_{v}=I_{v}[X]$ is a $v$-ideal of finite type and this
forces $I_{v}$ to be a $v$-ideal of finite type.
\end{proof}

\begin{acknowledgement}
There were a lot of typos and ambiguities in the original version, as usual.
I am indebted to Professor Anthony Bevelacqua for pointing them out to me.
Hopefully the paper is a better read now.
\end{acknowledgement}

\end{document}